\documentclass{article}
\usepackage{a4wide}
\usepackage{amsmath}
\usepackage{amssymb}
\usepackage{amsthm}
\usepackage{graphicx}
\usepackage{bbm}
\usepackage{enumerate}
\usepackage{microtype}
\usepackage[colorlinks=true,urlcolor=blue,linkcolor=blue,plainpages=false,pdfpagelabels]{hyperref}

\newtheorem{theorem}{Theorem}[section]
\newtheorem{proposition}[theorem]{Proposition}
\newtheorem{lemma}[theorem]{Lemma}
\newtheorem{corollary}[theorem]{Corollary}

\newtheorem{definition}[theorem]{Definition}

\newcommand{\N}{\mathbb{N}}
\newcommand{\eps}{\varepsilon}

\title{The asymptotic behaviour of convex combinations of firmly nonexpansive mappings}
\author{Andrei Sipo\c s${}^{a,b}$\\[2mm]
\footnotesize ${}^a$Department of Mathematics, Technische Universit\"at Darmstadt,\\
\footnotesize Schlossgartenstrasse 7, 64289 Darmstadt, Germany\\[1mm]
\footnotesize ${}^b$Simion Stoilow Institute of Mathematics of the Romanian Academy,\\
\footnotesize Calea Grivi\c tei 21, 010702 Bucharest, Romania \\[2mm]
\footnotesize E-mail: sipos@mathematik.tu-darmstadt.de\\
}
\date{}

\begin{document}

\maketitle

\begin{abstract}
We show that in the framework of CAT(0) spaces, any convex combination of two mappings which are firmly nonexpansive -- or which satisfy the more general property $(P_2)$ -- is asymptotically regular, conditional on its fixed point set being nonempty, and, in addition, also $\Delta$-convergent to such a fixed point. These results are established by the construction and study of a convex combination metric on the Cartesian square of a CAT(0) space. We also derive a uniform rate of asymptotic regularity in the sense of proof mining. All these results are then interpreted in the special case of the mappings being projections onto closed, convex sets.


\end{abstract}

\section{Introduction and preliminaries}

So far, a significant amount of research in nonlinear analysis has been driven by the search for iterative algorithms that approximate fixed points of self-mappings of (subsets of) spaces of various kinds, such as Banach or Hilbert spaces, the latter especially in convex optimization -- see, e.g. \cite{BauCom17}. More recently, the attention has turned towards nonlinear generalizations of such objects, i.e. metric spaces endowed with some additional structure or property that allows one to extend in a natural way arguments that were developed in the context of normed spaces. For example, we say that a metric space $(X,d)$ is {\it geodesic} if for any two points $x,y \in X$ there is a geodesic that joins them, i.e. a mapping $\gamma : [0,1] \to X$ such that $\gamma(0)=x$, $\gamma(1)=y$ and for any $t,t' \in [0,1]$ we have that
$$d(\gamma(t),\gamma(t')) = |t-t'| d(x,y).$$
A subclass of geodesic spaces that is usually considered to be the proper nonlinear analogue of Hilbert spaces is the class of CAT(0) spaces, introduced by A. Aleksandrov \cite{Ale51} and named as such by M. Gromov \cite{Gro87}, defined as those geodesic spaces $(X,d)$ such that for any geodesic $\gamma : [0,1] \to X$ and for any $z \in X$ and $t \in [0,1]$ we have that
$$d^2(z,\gamma(t)) \leq (1-t)d^2(z,\gamma(0)) + td^2(z,\gamma(1)) - t(1-t)d^2(\gamma(0),\gamma(1)).$$
It is known (see \cite[Theorem 6]{BerNik08}) that for a geodesic space $(X,d)$, the above condition is equivalent to the following: for all $x,y,z,w \in X$,
\begin{equation}
d^2(x,z) + d^2(y,w) \leq d^2(x,y) + d^2(y,z) + d^2(z,w) + d^2(w,x). \label{cat0-easy}
\end{equation}
Another well-known fact about CAT(0) spaces is that each such space $(X,d)$ is {\it uniquely geodesic} -- that is, for any $x,y \in X$ there is a unique such geodesic $\gamma : [0,1] \to X$ that joins them -- and in this context we shall denote, for any $t \in [0,1]$, the point $\gamma(t)$ by $(1-t)x + ty$.

The mappings that play a central role in convex optimization are the {\it firmly nonexpansive} ones, since they encompass a large number of concrete and useful cases such as proximal mappings or resolvents. The following nonlinear generalization of firmly nonexpansive mappings was introduced in \cite{AriLeuLop14}.

\begin{definition}\label{def-fn}
Let $(X,d)$ be a CAT(0) space. A mapping $T : X \to X$ is called {\em firmly nonexpansive} if for any $x,y \in X$ and any $t \in [0,1]$ we have that
$$d(Tx,Ty) \leq d((1-t)x + tTx,(1-t)y+tTy).$$
\end{definition}

A primary example of a firmly nonexpansive mapping is the following. Given a nonempty closed convex subset $C$ of a CAT(0) space $(X,d)$, we have, by \cite[Proposition II.2.4]{BriHae99}, that for any $x \in X$ there exists a unique $z \in C$ such that for any $w \in C$ we have that $d(x,z) \leq d(x,w)$. Such a point $z$ is called the {\em metric projection} of $x$ onto $C$ and is denoted by $P_Cx$. By \cite[Proposition 3.1]{AriLeuLop14}, the metric projection operator $P_C : X \to X$ is firmly nonexpansive.

As mentioned in \cite{AriLopNic15,KohLopNic17}, if $(X,d)$ is a CAT$(0)$ space, any firmly nonexpansive mapping $T:X \to X$ satisfies {\em property $(P_2)$}. Namely, for all $x,y \in X$,
\begin{equation}
2d^2(Tx,Ty) \leq d^2(x,Ty) + d^2(y,Tx) - d^2(x,Tx) - d^2(y,Ty). \label{p2}
\end{equation}
If $X$ is a Hilbert space, property $(P_2)$ is sufficient for firm nonexpansivity as the above property immediately yields $\|Tx - Ty\|^2 \le \langle Tx-Ty,x-y \rangle$, which is, in turn, equivalent to Definition~\ref{def-fn} (see, e.g., \cite[Proposition 4.4]{BauCom17} for a proof).

From now on, if $T$ is a self-mapping of some set, we denote by $Fix(T)$ the set of its fixed points. A fundamental concept in the fixed point theory of self-mappings is {\it asymptotic regularity}, introduced by Browder and Petryshyn \cite{BroPet66}. If $T$ is a self-mapping of a set that contains a sequence $(x_n)$, then $(x_n)$ is said to be {\it $T$-asymptotically regular} or to be an {\it approximate fixed point sequence} for $T$ if
$$\lim_{n \to \infty} d(x_n,Tx_n) = 0.$$
We may say of an arbitrary sequence $(x_n)$ that it is simply {\it asymptotically regular} if
$$\lim_{n \to \infty} d(x_n,x_{n+1}) = 0.$$
If the sequence is the {\it Picard iteration} corresponding to a self-mapping $T$ starting from a point $x$ -- that is, the sequence $(T^nx)$ -- then the two notions coincide. In fact, the original formulation called a mapping $T$ to be {\it asymptotically regular} if for any $x$ in its domain, the Picard iteration of $T$ starting from $x$ is an asymptotically regular sequence in the sense just defined.

By \cite{AriLeuLop14}, if $(X,d)$ is a CAT(0) space and $T: X \to X$ is a firmly nonexpansive mapping with $Fix(T) \neq \emptyset$, then $T$ is asymptotically regular; it is also an immediate consequence of Theorem~\ref{th-comp}, whose statement will be shown momentarily, that this holds also if $T$ only satisfies property $(P_2)$. However, these classes of functions are not necessarily closed under composition (not even in Hilbert spaces where the two classes coincide -- see, e.g. \cite[Example 4.45]{BauCom17}), so we may naturally ask ourselves if the property continues to hold if $T$ is a composition of finitely many firmly nonexpansive mappings. In Hilbert spaces, this is an old result, going back to Bruck and Reich \cite{BruRei77}, who proved it for the more general case of $T$ being strongly nonexpansive. In CAT(0) spaces, on the other hand, we only have the following recent result for the composition of two mappings.

\begin{theorem}[cf. {\cite[Theorem 3.3]{AriLopNic15}}]\label{th-comp}
Let $(X,d)$ be a CAT(0) space and let $T_1, T_2 : X \to X$ satisfy property $(P_2)$. Put $T: = T_2 \circ T_1$. Assuming that $Fix(T) \neq \emptyset$, we have that $T$ is asymptotically regular.
\end{theorem}

Such a Picard iteration, therefore, aims to approximate fixed points of $T$. In the special case of $Fix(T_1) \cap Fix(T_2)$ being nonempty, it is known (by \cite[Proposition 2.1]{AriLopNic15}) that this set coincides with $Fix(T)$ and we speak in this case of a {\it consistent feasibility problem}. The case studied above may be located between this and a more general version where we only know that each $Fix(T_i)$ is nonempty. Such a problem has been already studied in \cite{Bau03, BauMarMofWan12} from the viewpoint of asymptotic regularity, though in the setting of Hilbert spaces, and it is commonly known as a problem of {\it inconsistent feasibility}.

In the midway situation elaborated upon in the above -- which we may therefore dub the problem of {\it intermediate feasibility} -- it is also relevant to ask whether the fixed points of $T$ are approximated, in addition, from the viewpoint of convergence. Since a large amount of iterative algorithms that are used for this purpose in the context of normed spaces only weakly converge to their desired target, we would not expect strong convergence to hold in this more general geodesic setting. Therefore, one needs a suitable analogue of the concept of weak convergence.

The notion of $\Delta$-convergence was first defined by Lim \cite{Lim76} in metric spaces, and some equivalent notions in the setting of complete CAT(0) spaces may be found in \cite{EspFer09,Jos94,Kuc80}. It is indeed a proper generalization of weak convergence, since it coincides with the latter in Hilbert spaces, as per \cite[Exercise 3.1]{Bac14}; see also \cite{KirPan08} for related results in connection to Banach spaces. To define this notion, we first need to define the auxiliary one of {\it asymptotic center}. Let $(x_n)$ be a bounded sequence in a complete CAT(0) space $X$. Define
$$r((x_n)):= \inf_{y \in X} \limsup_{n \to \infty} d(y,x_n)$$
and then put $A((x_n)) := \{c \in X \mid \limsup_{n \to \infty} d(c,x_n) = r((x_n))\}$. The elements of this latter set are called {\em asymptotic centers} of $(x_n)$. It is known (see \cite[Proposition 7]{DhoKirSim06}) that every bounded sequence $(x_n)$ in a complete CAT(0) space $X$ has a unique asymptotic center. That being said, $\Delta$-convergence is defined as follows.

\begin{definition}
A bounded sequence $(x_n)$ {\em $\Delta$-converges} to a point $x \in X$ if for any subsequence 
$(x_{n_k})$ of $(x_n)$ we have that $A((x_{n_k}))=\{x\}$.
\end{definition}

Now we may properly state the convergence theorem obtained in \cite{AriLopNic15}.

\begin{theorem}[cf. {\cite[Theorem 4.2]{AriLopNic15}}]\label{th-deltaconv}
Let $(X,d)$ be a complete CAT(0) space and let $T_1, T_2 : X \to X$ satisfy property $(P_2)$. Put $T: = T_2 \circ T_1$. Assuming that $Fix(T) \neq \emptyset$, we have that for any $x \in X$, the Picard iteration of $T$ starting with $x$ $\Delta$-converges to a fixed point of $T$.
\end{theorem}

The primary application of this kind of algorithm is the case where the two mappings are the projection operators on two closed, convex, nonempty subsets $A$ and $B$ of the space $X$. Let us first fix some ideas. Define the distance between $A$ and $B$ as
$$d(A,B):=\inf_{(x,y) \in A \times B} d(x,y),$$
and denote by $S_{A,B}$ the set of all pairs $(x^*,y^*)$ such that $d(x^*,y^*)=d(A,B)$, which are called {\it best approximation pairs}. It is immediate that if $(a,b) \in S_{A,B}$ then $b = P_Ba$ and $a = P_Ab$; in particular, if $S_{A,B} \neq \emptyset$ then $Fix(P_A \circ P_B)$ is also nonempty. It is those best approximation pairs that one aims to find, as the following corollary to the result above shows us.

\begin{corollary}[cf. {\cite[Corollary 4.3]{AriLopNic15}}]\label{cor-deltaconv}
Let $(X,d)$ be a complete CAT(0) space and let $A$ and $B$ be two convex, closed, nonempty subsets of $X$. Set $T:= P_A \circ P_B$. Assuming that $S_{A,B} \neq \emptyset$, we have that for any $x \in X$ there is a pair $(a,b) \in S_{A,B}$ -- and therefore $b=P_Ba$ -- such that the Picard iteration of $T$ starting with $x$ $\Delta$-converges to the point $a$.
\end{corollary}

The above concrete variant of the algorithm is usually known as the {\it alternating projections method}. A related one, the {\it method of averaged projections}, will act as the focus of this paper. This method replaces the composition of projections from above with a weighted average of them, i.e. a convex combination. A classical way of proving the efficacy of this method in the usual Euclidean space is by reducing it to the first one through the replacing of the combined mapping with a concatenated transformation in a product space with the averaged metric, followed by a projection on the diagonal. This technique was recently extended to Hilbert spaces in \cite[Section 5]{BauMarMofWan12}, where it was used to prove an analogue for the abovementioned asymptotic regularity result for the inconsistent feasibility problem. (As an aside, we remark that the case of nonzero minimal displacement considered in \cite{BauMou18} falls beyond the known range of this trick.) In the analogue of the case considered before, where the fixed point set of the averaged mapping is nonempty, the corresponding result is trivial in Hilbert spaces, where any convex combination of finitely many firmly nonexpansive mappings is also firmly nonexpansive (see, e.g., \cite[Example 4.7]{BauCom17}). It is not known yet whether this holds in CAT(0) spaces either for firmly nonexpansive mappings or for mappings satisfying property $(P_2)$. Hence, it is nontrivial to ask ourselves whether the above result for two mappings continues to hold if we replace composition by convex combination, and whether the same kind of reduction technique would be also usable in this framework.

{\it The goal of this paper is to extend that line of argument to self-mappings satisfying property $(P_2)$ -- in particular, to firmly nonexpansive self-mappings -- in the setting of CAT(0) spaces.}

To that degree, in Section~\ref{sec:cc} we will introduce the needed construction -- a distance function $d_\lambda$ on the Cartesian square of a CAT(0) space that makes it into a new CAT(0) space which may be regarded as a weighted average. We shall prove that it has the required properties that allow it to behave well with respect to standard constructions regarding this kind of iterative method. In addition, we show how the best approximation pairs and the $\Delta$-convergent sequences in the averaged space relate to those in the original one. In Section~\ref{sec:mr}, we apply these lemmas to obtain the analogues of the results above. Theorem~\ref{main} establishes the asymptotic regularity of the averaged mapping, while Theorem~\ref{th-deltaconv2} yields its $\Delta$-convergence to a fixed point of it. In particular, Corollary~\ref{cor-deltaconv2} shows that in the specific case of averaged projections, the iteration $\Delta$-converges to the convex combination, by the same parameter $\lambda$, of a best approximation pair.

In addition, we will also establish effective versions of such asymptotic results, in the sense of proof mining. Proof mining is an applied subfield of mathematical logic, developed primarily by U. Kohlenbach and his collaborators (\cite{Koh08} is the standard introduction, while more recent surveys are \cite{Koh17, Koh18}), that aims to provide quantitative information (witnesses and bounds) for numerical entities which are shown to exist by proofs which cannot be necessarily said to be fully constructive. In nonlinear analysis, which has been a primary focus for such work, the relevant information is usually found within convergence statements, where the problem is to find an explicit formula for the $N_\eps$ such that for any $\eps > 0$, the elements of the sequence of index greater than $N_\eps$ are $\eps$-close to the limit. Theorem~\ref{th-comp} above is an instance of this, as asymptotic regularity is clearly a statement of convergence, and it has indeed been analyzed from this viewpoint in \cite{KohLopNic17}, yielding the following quantitative version.

\begin{theorem}[cf. {\cite[Theorem 3.1]{KohLopNic17}}]\label{th-comp-rate}
Let $(X,d)$ be a CAT(0) space and let $T_1, T_2 : X \to X$ satisfy property $(P_2)$. Put $T: = T_2 \circ T_1$. Let $p \in Fix(T)$ and $b > 0$. Define, for any $\eps>0$,
$$k_b(\eps):= \left\lceil \frac{2b}{\eps} \right \rceil, \qquad \Phi_b(\eps):=k_b(\eps) \cdot \left\lceil \frac{2b(1+2^{k_b(\eps)})}{\eps} \right \rceil + 1.$$
Then, for any $x \in X$ with $d(x,p) \leq b$, we have that
$$\forall \eps > 0\ \forall n \geq \Phi_b(\eps)\ d(T^nx, T^{n+1}x) \leq \eps.$$
\end{theorem}

In the above, $\Phi_b(\eps)$, being the formula for $N_\eps$, is called a {\it rate of asymptotic regularity}. A feature that it exhibits and that is characteristic to the products of proof mining is that it is highly uniform, i.e. it is independent of parameters like the space or mapping involved. Such a rate may be deemed useful as one needs to know how many times should the algorithm be iterated in order to obtain an approximate fixed point within some pre-specified precision. The proof of corresponding result for the convex combination case, Proposition~\ref{main-rate}, shows that the rate is left unchanged during the reduction procedure.

In contrast, the special case of averaged projections is a bit more involved. In the composition case, the required information is how close the current step of the iteration is to being a best approximation pair, i.e. how close the distance between it and the projection on the second set is to the distance between the two given sets. This problem had been already analyzed in \cite{AriLopNic15}, where the following result is implicit.

\begin{corollary}[cf. {\cite[Corollary 3.3 and the discussion following it]{AriLopNic15}}]\label{quant-best-approx}
Let $(X,d)$ be a CAT(0) space and let $A$ and $B$ be two convex, closed, nonempty subsets of $X$. Set $T:=P_A \circ P_B$ and
$$q:=d^2(A,B) = \inf_{(x,y) \in A \times B} d^2(x,y).$$
Assume that there is a pair $(x^*,y^*) \in A \times B$ with $d^2(x^*,y^*) = q$. Let $M,b >0$.

Then, for any $x \in X$ with $d(x,x^*) \leq M$ and $d^2(P_AP_Bx,P_Bx)\leq b$, we have that
$$\forall \eps > 0\ \forall n \geq \left(\left\lfloor \frac{4M^2b}{\eps^2} \right\rfloor + 2\right)\ d^2(T^nx, P_BT^nx) \leq q+ \eps.$$
\end{corollary}

In order to directly apply the above result to obtain the analogue one for convex combinations, we shall need an additional preliminary lemma. Specifically, one not only needs to relate best approximation pairs in the averaged space to the ones in the original space, but also approximately-close ones. The good news is that techniques of proof mining allow us to analyze a set inclusion in order to derive its approximate version. This is what we do in Lemma~\ref{quant-incl}, and this is what we apply in the proof of Corollary~\ref{rate-proj} to obtain our desired rate.

\section{The convex combination of two CAT(0) spaces}\label{sec:cc}

As announced in the Introduction, we shall now define and establish some preliminary results for the convex combination metric on $X \times X$, an analogue of the convex combination Hilbert space structure introduced in \cite{BauMarMofWan12}.

If $(X,d)$ is a metric space and $\lambda \in (0,1)$ we introduce the function $d_\lambda : X^2 \times X^2 \to \mathbb{R}_+$, defined for any $(x_1,x_2), (y_1,y_2) \in X^2$ by:
$$d_\lambda((x_1,x_2),(y_1,y_2)) := \sqrt{(1-\lambda)d^2(x_1,y_1) + \lambda d^2(x_2,y_2)}.$$

\begin{proposition}
Let $(X,d)$ be a metric space and $\lambda \in (0,1)$. Then:
\begin{enumerate}[(i)]
\item $(X^2,d_\lambda)$ is a metric space;
\item if $(X,d)$ is complete, then $(X^2,d_\lambda)$ is also complete;
\item if $(X,d)$ is a geodesic space, then $(X^2,d_\lambda)$ is also geodesic;
\item if $(X,d)$ is a CAT(0) space, then $(X^2,d_\lambda)$ is also CAT(0).
\end{enumerate}
\end{proposition}

\begin{proof}
\begin{enumerate}[(i)]
\item This is a classical result which may be found, usually as an exercise, in any basic reference on metric spaces. For example, one may first scale the metric twice, by $\sqrt{1-\lambda}$ and by $\sqrt{\lambda}$, and then consider \cite[Section 4.1, Example 3]{Kap77}.

\item This is immediate.

\item Let $(x_1,x_2)$, $(y_1,y_2) \in X^2$. Let, for each $i \in \{1,2\}$, $\gamma_i : [0,1] \to X$ be a geodesic from $x_i$ to $y_i$. Define $\gamma : [0,1] \to X^2$, for any $t \in [0,1]$, by:
$$\gamma(t) := (\gamma_1(t),\gamma_2(t)).$$
We shall now show that $\gamma$ is a geodesic. Clearly $\gamma(0) = (x_1,x_2)$ and $\gamma(1) = (y_1,y_2)$. Let now $t, t' \in [0,1]$. Then,
\begin{align*}
d_\lambda(\gamma(t),\gamma(t')) &= \sqrt{(1-\lambda) d^2(\gamma_1(t), \gamma_1(t')) + \lambda d^2(\gamma_2(t), \gamma_2(t'))} \\
&= \sqrt{(1-\lambda) (t-t')^2 d^2(x_1,y_1) + \lambda (t-t')^2 d^2(x_2,y_2)} \\
&= |t-t'| \sqrt{(1-\lambda) d^2(x_1,y_1) + \lambda d^2(x_2,y_2)} \\
&= |t-t'| d_\lambda((x_1,x_2),(y_1,y_2)).
\end{align*}

\item This is an immediate consequence of (iii) and \eqref{cat0-easy}.
\end{enumerate}
\end{proof}

The following proposition may be easily proven from the defining equation of the property $(P_2)$, namely \eqref{p2}.

\begin{proposition}\label{prop23}
Let $(X,d)$ be a metric space and $\lambda \in (0,1)$. Consider two mappings $T_1, T_2: X\to X$ that satisfy property $(P_2)$. Define a mapping $U: X^2 \to X^2$ by setting, for any $(x_1,x_2) \in X^2$,
$$U(x_1,x_2) := (T_1x_1,T_2x_2).$$
Then $U$ satisfies property $(P_2)$ with respect to $(X^2, d_\lambda)$.
\end{proposition}

The other mapping that will be used in our application of Theorem~\ref{th-comp} is provided by the next proposition.

\begin{proposition}\label{prop24}
Let $(X,d)$ be a CAT(0) space and $\lambda \in (0,1)$. Consider the set
$$\Delta_X := \{(x,x) \mid x \in X\} \subseteq X^2,$$
which is clearly closed and convex, and define the operator $Q : X^2 \to X^2$ by setting, for any $(x_1,x_2) \in X^2$,
$$Q(x_1,x_2) := ((1-\lambda)x_1+\lambda x_2,(1-\lambda)x_1+\lambda x_2).$$
Then $Q$ is the metric projection operator onto $\Delta_X$ in $(X^2,d_\lambda)$.
\end{proposition}

\begin{proof}
What we have to show is that for any $w, x_1, x_2 \in X$,
$$d_\lambda^2((x_1,x_2),((1-\lambda)x_1+\lambda x_2,(1-\lambda)x_1+\lambda x_2)) \leq d_\lambda^2((x_1,x_2),(w,w)).$$
Let $w, x_1, x_2 \in X$. Set $c:= (1-\lambda)x_1+\lambda x_2$. By the defining property of CAT(0) spaces, we have that
$$d^2(w,c) \leq (1-\lambda)d^2(w,x_1) + \lambda d^2(w,x_2) - \lambda(1-\lambda) d^2(x_1,x_2)$$
and, therefore, since $d^2(w,c) \geq 0$,
\begin{equation}
\lambda(1-\lambda) d^2(x_1,x_2) \leq (1-\lambda)d^2(w,x_1) + \lambda d^2(w,x_2). \label{lam-unu}
\end{equation}
We may now compute:
\begin{align*}
d_\lambda^2((x_1,x_2),(c,c)) &= (1-\lambda)d^2(x_1,c) + \lambda d^2(x_2,c) \\
&= \lambda^2(1-\lambda)d^2(x_1,x_2) + \lambda(1-\lambda)^2d^2(x_1,x_2) \\
&= \lambda(1-\lambda)d^2(x_1,x_2) \\
&\leq (1-\lambda)d^2(x_1,w) + \lambda d^2(x_2,w) & \text{by \eqref{lam-unu}} \\
&= d_\lambda^2((x_1,x_2),(w,w)),
\end{align*}
which was what we needed to show.
\end{proof}

We will also need the following lemmas.

\begin{lemma}\label{fixqu}
Let $(X,d)$ be a CAT(0) space and let $T_1, T_2: X\to X$ satisfy property $(P_2)$. Let $\lambda \in(0,1)$ and put $T:=(1-\lambda)T_1 + \lambda T_2$. Define $U$ as in Proposition~\ref{prop23} and $Q$ as in Proposition~\ref{prop24}. Then:
\begin{enumerate}[(i)]
\item $Fix(Q \circ U) = \{(p,p) \mid p \in Fix(T)\}$;
\item for all $x,y \in X$, $d_\lambda((x,x),(y,y)) = d(x,y)$;
\item for all $x \in X$ and all $n \in \N$, $(Q \circ U)^n(x,x) = (T^nx,T^nx)$.
\end{enumerate}
\end{lemma}

\begin{proof}
\begin{enumerate}[(i)]
\item Let $p$ be a fixed point of $T$. Then $(p,p)$ is a fixed point of $Q \circ U$, since
$$QU(p,p) = Q(T_1p,T_2p) =((1-\lambda)T_1p+\lambda T_2p,(1-\lambda)T_1p+\lambda T_2p) = (p,p).$$

Conversely, suppose that $(x_1,x_2)$ is a fixed point of $Q \circ U$. Then, since
$$((1-\lambda)T_1x_1 + \lambda T_2x_2, (1-\lambda)T_1x_1 + \lambda T_2x_2)=(x_1,x_2),$$
we have that $x_1 = x_2 =: p$ and therefore that $(1-\lambda)T_1p + \lambda T_2p = p$.
\item Immediate, using the definition of $d_\lambda$.
\item Immediate, by induction on $n$.
\end{enumerate}
\end{proof}

\begin{lemma}\label{quant-incl}
Let $(X,d)$ be a CAT(0) space and $\lambda \in (0,1)$. Let $\eps \geq 0$. Let $A, B$ be two subsets of $X$ and $w \in X$, $a \in A$, $b \in B$ be such that for all $x \in X$, $y \in A$ and $z \in B$,
$$d_\lambda^2((w,w),(a,b)) \leq d_\lambda^2((x,x),(y,z)) + \frac{\eps^2}4 \cdot \lambda(1-\lambda).$$
Then, for any $y \in A$, $z\in B$,
$$d(a,b) \leq d(y,z) + \eps.$$
\end{lemma}

\begin{proof}
First, set $u:=(1-\lambda)a+\lambda b$.\\[1mm]

\noindent {\bf Claim:} We have that $d^2(w,u) \leq \frac{\eps^2}4 \cdot \lambda(1-\lambda)$.\\
{\bf Proof of claim:} We apply our hypothesis for $x:=u$, $y:=a$, $z:=b$ to obtain that
$$(1-\lambda)d^2(w,a) + \lambda d^2(w,b) \leq \lambda(1-\lambda) d^2(a,b) + \frac{\eps^2}4 \cdot \lambda(1-\lambda).$$
Since, by the defining property of CAT(0) spaces, we have that
$$d^2(w,u) \leq (1-\lambda)d^2(w,a) + \lambda d^2(w,b) - \lambda(1-\lambda) d^2(a,b),$$
the conclusion follows. \hfill $\blacksquare$\\[1mm]

Now let $y\in A$ and $z \in B$ and put $v:=(1-\lambda)y+\lambda z$. Applying the hypothesis, we get:
$$d_\lambda((w,w),(a,b))\leq d_\lambda((v,v),(y,z)) + \frac\eps2 \cdot \sqrt{\lambda(1-\lambda)}.$$
On the other hand,
\begin{align*}
d_\lambda((u,u),(a,b)) &\leq d_\lambda((u,u),(w,w)) + d_\lambda((w,w),(a,b)) \\
&\leq \frac\eps2 \cdot \sqrt{\lambda(1-\lambda)} + d_\lambda((v,v),(y,z)) + \frac\eps2 \cdot \sqrt{\lambda(1-\lambda)}.
\end{align*}
Therefore, we get that:
$$\sqrt{\lambda(1-\lambda)} d(a,b) \leq \sqrt{\lambda(1-\lambda)} d(y,z) + \eps\sqrt{\lambda(1-\lambda)},$$
from which we obtain our conclusion.
\end{proof}

The above lemma, as announced before, is the approximate version of one half of the following lemma, as it may be immediately seen in the latter's proof.

\begin{lemma}\label{sdelta}
Let $(X,d)$ be a CAT(0) space and $\lambda \in (0,1)$. Let $A, B$ be two subsets of $X$. Then
$$S_{\Delta_X, A \times B} = \{ (((1-\lambda)a+\lambda b,(1-\lambda)a+\lambda b),(a,b)) \mid (a,b) \in S_{A,B} \},$$
where the left hand side set is considered as being defined with respect to $(X^2,d_\lambda)$.
\end{lemma}

\begin{proof}
The forward inclusion is an immediate consequence of Lemma~\ref{quant-incl} and of the claim in its proof, by taking $\eps:=0$. We shall now prove the backward inclusion.

Take $(a,b) \in S_{A,B}$ and set $u:=(1-\lambda)a+\lambda b$.

We know that for all $y \in A$ and $z \in B$, $d(a,b) \leq d(y,z)$. Let $x \in X$, $y \in A$ and $z \in B$. We want to show that
$$d_\lambda^2((u,u),(a,b)) \leq d_\lambda^2((x,x),(y,z)),$$
i.e., that
$$\lambda(1-\lambda)d^2(a,b) \leq (1-\lambda)d^2(x,y) + \lambda d^2(x,z).$$
Set now $w:=(1-\lambda)y+\lambda z$. By the defining property of CAT(0) spaces, we have that
$$0 \leq d^2(x,w) \leq (1-\lambda) d^2(x,y) + \lambda d^2(x,z) - \lambda(1-\lambda) d^2(y,z),$$
and hence that
$$\lambda(1-\lambda) d^2(y,z) \leq (1-\lambda) d^2(x,y) + \lambda d^2(x,z).$$
Since $d(a,b) \leq d(y,z)$, our conclusion follows.
\end{proof}

Finally, we relate the notion of $\Delta$-convergence in the averaged space to the one in the original space.

\begin{lemma}\label{doubledelta}
Let $(X,d)$ be a complete CAT(0) space and $(x_n)$ be a bounded sequence in $X$. Let $u \in X$ and assume that $((x_n,x_n))$ $\Delta$-converges to $(u,u)$ in $(X^2,d_\lambda)$. Then $(x_n)$ $\Delta$-converges to $u$.
\end{lemma}

\begin{proof}
Let $(x_{n_k})$ be a subsequence of $(x_n)$. We want to show that $A((x_{n_k}))=\{u\}$. By the uniqueness of asymptotic centers, it suffices to show that $u$ is an asymptotic center of $(x_{n_k})$. Since we know that $A((x_{n_k},x_{n_k}))=\{(u,u)\}$, we have that for all $y \in X$,
$$\limsup_{k \to \infty} d_\lambda((u,u),(x_{n_k},x_{n_k})) \leq \limsup_{k \to \infty} d_\lambda((y,y),(x_{n_k},x_{n_k})),$$
and hence, by Lemma~\ref{fixqu}.(ii), that for all $y \in X$,
$$\limsup_{k \to \infty}d(u,x_{n_k}) \leq \limsup_{k \to \infty} d(y,x_{n_k}),$$
which yields our conclusion.
\end{proof}

\section{Asymptotic results}\label{sec:mr}

We will begin by proving the results on asymptotic regularity, both qualitative and quantitative, which will be then followed by the $\Delta$-convergence ones. The following is the analogue of Theorem~\ref{th-comp}.

\begin{theorem}\label{main}
Let $(X,d)$ be a CAT(0) space and let $T_1, T_2 : X \to X$ satisfy property $(P_2)$. Let $\lambda \in(0,1)$ and put $T:=(1-\lambda)T_1 + \lambda T_2$. Assuming that $Fix(T) \neq \emptyset$, we have that $T$ is asymptotically regular.
\end{theorem}

\begin{proof}
Let $x \in X$ and put, for any $n \in \N$, $x_n:=T^nx$. We have to show that
$$ \lim_{n \to \infty}d(x_n,x_{n+1}) = 0.$$
Define $U$ as in Proposition~\ref{prop23} and $Q$ as in Proposition~\ref{prop24}. By Lemma~\ref{fixqu}.(i), $Fix(Q \circ U) \neq \emptyset$.

Denote by $(\vec{x}_n)$ the Picard iteration of $Q \circ U$ starting from $(x,x)$. Since $Q$ is, by Proposition~\ref{prop24}, a metric projection operator in a CAT(0) space, it is firmly nonexpansive and hence satisfies property $(P_2)$. We may then apply Theorem~\ref{th-comp} to obtain that
\begin{equation}
\lim_{n \to \infty} d_\lambda(\vec{x}_n,\vec{x}_{n+1})=0. \label{asreg-vec}
\end{equation}

Since, by Lemma~\ref{fixqu}.(iii), for any $n$, $\vec{x}_n = (x_n,x_n)$, we get that \eqref{asreg-vec}, together with Lemma~\ref{fixqu}.(ii), entails
$$ \lim_{n \to \infty}d(x_n,x_{n+1}) = 0,$$
which is what was needed to be shown.
\end{proof}

In the light of the above proof, it is hopefully now clear that the rate of asymptotic regularity previously obtained for the composition case also holds for the convex combination case. The following proposition and proof make this completely explicit.

\begin{proposition}\label{main-rate}
Let $(X,d)$ be a CAT(0) space and let $T_1, T_2 : X \to X$ satisfy property $(P_2)$. Let $\lambda \in (0,1)$ and put $T:=(1-\lambda)T_1 + \lambda T_2$. Let $p \in Fix(T)$ and $b > 0$. Define, for any $\eps>0$, $\Phi_b(\eps)$ as in Theorem~\ref{th-comp-rate}.

Then, for any $x \in X$ with $d(x,p) \leq b$, we have that
$$\forall \eps > 0\ \forall n \geq \Phi_b(\eps)\ d(T^nx, T^{n+1}x) \leq \eps.$$
\end{proposition}

\begin{proof}
Let $x \in X$ with $d(x,p) \leq b$. We shall use the notation (and some results obtained in this framework) from the proof of Theorem~\ref{main}. Since, by Lemma~\ref{fixqu}.(i), $(p,p)$ is a fixed point of $Q \circ U$ and, by Lemma~\ref{fixqu}.(ii),
$$d_\lambda((x,x),(p,p)) = d(x,p) \leq b,$$
we may apply Theorem~\ref{th-comp-rate} to obtain that
$$\forall \eps > 0\ \forall n \geq \Phi_b(\eps)\ d_\lambda(\vec{x}_n,\vec{x}_{n+1}) \leq \eps.$$
But this immediately yields our result since, as before, we have, again by Lemma~\ref{fixqu}, (ii) and (iii), that for all $n$, $d_\lambda(\vec{x}_n,\vec{x}_{n+1}) = d(T^nx, T^{n+1}x)$.
\end{proof}

We will now use Lemma~\ref{quant-incl} to obtain a quantitative result in the case of projections.

\begin{corollary}\label{rate-proj}
Let $(X,d)$ be a CAT(0) space and let $A$ and $B$ be two convex, closed, nonempty subsets of $X$. Let $\lambda \in (0,1)$ and put $T:=(1-\lambda)P_A + \lambda P_B$. Set
$$r:=d(A,B) = \inf_{(x,y) \in A \times B} d(x,y).$$
Let $(x^*,y^*) \in S_{A,B}$ and $M,b>0$. Set $u^*:=(1-\lambda)x^* + \lambda y^*$.

Then, for any $x\in X$ with $d(x,u^*) \leq M$ and $d^2(P_Ax,P_Bx) \leq b$, we have that 
$$\forall \eps > 0\ \forall n \geq \left(\left\lfloor \frac{64M^2b}{\eps^4\lambda(1-\lambda)} \right\rfloor + 2\right)\ d(P_AT^nx, P_BT^nx) \leq r+ \eps.$$
\end{corollary}

\begin{proof}
It is clear that $\Delta_X$ and $A \times B$ are two closed, convex, nonempty subsets of $(X^2,d_\lambda)$. Set $q:=d^2(\Delta_X, A \times B)$. (Incidentally, by Lemma~\ref{sdelta}, we have that $q=\lambda(1-\lambda)r^2$.)

Define $U$ as in Proposition~\ref{prop23} and $Q$ as in Proposition~\ref{prop24} (for $T_1:=P_A$ and $T_2:=P_B$). Denote by $(\vec{x}_n)$ the Picard iteration of $Q \circ U$ starting from $(x,x)$. By Proposition~\ref{prop24}, we also have that $Q=P_{\Delta_X}$ and it is immediate that $U=P_{A\times B}$.

By Lemma~\ref{sdelta}, we have that $((u^*,u^*),(x^*,y^*)) \in S_{\Delta_X,A\times B}$. By our hypothesis, we also have that $d((x,x),(u^*,u^*)) \leq M$ and that
$$d_\lambda^2((Tx,Tx),(P_Ax,P_Bx)) = \lambda(1-\lambda) d^2(P_Ax,P_Bx) \leq \lambda(1-\lambda) b.$$
Let $\eps > 0$. We may now apply Corollary~\ref{quant-best-approx} to obtain that 
$$\forall n \geq \left(\left\lfloor \frac{64M^2b}{\eps^4\lambda(1-\lambda)} \right\rfloor + 2\right)\ d^2_\lambda((QU)^n(x,x),U(QU)^n(x,x)) \leq q+ \frac{\eps^2}4 \cdot \lambda(1-\lambda).$$
Let, therefore, $n \geq \left\lfloor \frac{64M^2b}{\eps^4\lambda(1-\lambda)} \right\rfloor + 2$. Then, for all $x \in X$, $y \in A$ and $z \in B$,
$$ d^2_\lambda((T^nx,T^nx),(P_AT^nx,P_BT^nx)) \leq d^2_\lambda((x,x),(y,z))+ \frac{\eps^2}4 \cdot \lambda(1-\lambda).$$
We apply Lemma~\ref{quant-incl} to get that for all $y \in A$ and $z \in B$,
$$d(P_AT^nx,P_BT^nx) \leq d(y,z) + \eps,$$
i.e., that
$$d(P_AT^nx,P_BT^nx) \leq r + \eps.$$
\end{proof}

Finally, we proceed to derive the corresponding analogues of the results on $\Delta$-convergence.

\begin{theorem}\label{th-deltaconv2}
Let $(X,d)$ be a complete CAT(0) space and let $T_1, T_2 : X \to X$ satisfy property $(P_2)$. Let $\lambda \in (0,1)$ and put $T: = (1-\lambda)T_1 + \lambda T_2$. Assuming that $Fix(T) \neq \emptyset$, we have that for any $x \in X$, the Picard iteration of $T$ starting with $x$ $\Delta$-converges to a fixed point of $T$.
\end{theorem}

\begin{proof}
Define $U$ as in Proposition~\ref{prop23} and $Q$ as in Proposition~\ref{prop24}. By Lemma~\ref{fixqu}.(i), $Fix(Q \circ U) \neq \emptyset$. Denote by $(\vec{x}_n)$ the Picard iteration of $Q \circ U$ starting from $(x,x)$. As in the proof of Theorem~\ref{main}, we know from Proposition~\ref{prop24} that $Q$ satisfies property $(P_2)$.

We may now apply Theorem~\ref{th-deltaconv} to get that there is an $u \in Fix(Q \circ U)$ such that $(\vec{x}_n)$ $\Delta$-converges to $u$. By Lemma~\ref{fixqu}, (i) and (iii), we have that there is a $p \in Fix(T)$ such that $((T^nx,T^nx))$ $\Delta$-converges to $(p,p)$. Our conclusion follows using Lemma~\ref{doubledelta}.
\end{proof}

\begin{corollary}\label{cor-deltaconv2}
Let $(X,d)$ be a complete CAT(0) space and let $A$ and $B$ be two convex, closed, nonempty subsets of $X$. Let $\lambda \in (0,1)$ and put $T: = (1-\lambda)P_A + \lambda P_B$. Assuming that $S_{A,B} \neq \emptyset$, we have that for any $x \in X$ there is a pair $(a,b) \in S_{A,B}$ such that the Picard iteration of $T$ starting with $x$ $\Delta$-converges to $p:=(1-\lambda)a + \lambda b$.
\end{corollary}

\begin{proof}
As in the proof of Corollary~\ref{rate-proj}, it is clear that $\Delta_X$ and $A \times B$ are two closed, convex, nonempty subsets of $(X^2,d_\lambda)$. Define $U$ as in Proposition~\ref{prop23} and $Q$ as in Proposition~\ref{prop24} (for $T_1:=P_A$ and $T_2:=P_B$). Denote by $(\vec{x}_n)$ the Picard iteration of $Q \circ U$ starting from $(x,x)$. Again, as before, by Proposition~\ref{prop24}, we also have that $Q=P_{\Delta_X}$ and it is immediate that $U=P_{A\times B}$.

By Lemma~\ref{sdelta}, we know that $S_{\Delta_X,A\times B} \neq \emptyset$. We may apply Corollary~\ref{cor-deltaconv} and Lemma~\ref{sdelta} to get that there is a pair $(a,b) \in S_{A,B}$ such that $(\vec{x}_n)$ $\Delta$-converges to $(p,p)$, where we have set $p:=(1-\lambda)a + \lambda b$. By Lemma~\ref{fixqu}.(iii), $((T^nx,T^nx))$ $\Delta$-converges to $(p,p)$. Our conclusion follows using Lemma~\ref{doubledelta}.
\end{proof}

\section{Acknowledgements}

I would like to thank Ulrich Kohlenbach for suggesting to me this problem to work on. I am also grateful to Genaro L\'opez-Acedo and to Adriana Nicolae for their valuable geometric insights.

This work has been supported by the German Science Foundation (DFG Project KO 1737/6-1).

\end{document}